\documentclass[11pt]{article}

\usepackage{amsmath,amsthm,amssymb,graphics}
\usepackage{mathrsfs}
\usepackage{fullpage}
\usepackage{amsfonts}
\usepackage{epsfig}
\usepackage{color}
\newtheorem{thm}{Theorem}[section]

\newtheorem{lem}[thm]{Lemma}
\theoremstyle{definition}
\newtheorem{defn}[thm]{Definition}

\theoremstyle{remark} \numberwithin{equation}{section}

\newcommand{\R}{\mathbb{R}}
\newcommand{\N}{\mathbb{N}}

\newcommand{\mathsym}[1]{{}}

\begin{document}
\setcounter{page}{1}

\title{\textbf{On the existence-uniqueness and exponential estimate for solutions to stochastic functional differential
equations driven by G-L\'evy process}}
\author{\textbf{Faiz Faizullah$^*$, Muhammad Farooq$^1$, MA Rana$^1$, Rahman Ullah$^2$  } \textbf{}
 \vspace{0.1cm}\\
{$^*$  Department of BS\&H, College of E\&ME, National University of Sciences}\\{ and Technology (NUST),
Pakistan}\\
{$^1$ Department of Mathematics and Statistics, Riphah International
}\\ {University Islamabad, Pakistan}\\
{$^2$ Department of Mahematics, Women University Swabi,
Pakistan}
}
\maketitle
\begin{abstract}
The existence-uniqueness theory for solutions to stochastic dynamic systems
is always a significant theme and has received a huge attention. The objective of
this article is to study the mentioned theory for stochastic functional differential
equations (SFDEs) driven by G-L\'evy process. The existence-uniqueness theorem for solutions to SFDEs driven by G-L\'evy process has been determined.
The error estimation between
the exact solution and Picard approximate solutions has been shown. In addition,
the exponential estimate has been derived.
\end{abstract}
\section{Introduction}
Stochastic dynamic equations based on G-Brownian motion have
been studied by several authors \cite{bl,fb,g,ly1,zxk}. Among them, the existence-uniqueness, stability, moment estimates, continuity and differentiability  properties of solution with respect to the initial data were studied in detail \cite{f2019, f3,ly2, p, rbs,rf,zlz}. Stochastic differential equations based on L\'evy
process perform a leading role in a broad range of applications, containing financial mathematics for describing the observed reality of financial markets \cite{ct}, physics for various phenomenons \cite{w}, genetics for the movement designs of many animals \cite{fe} and biology for modeling the spread of diseases \cite{jowh}. In \cite{hp1}  Hu and  Peng initiated the G-L\'evy process. In \cite{r} Ren represented a sublinear expectation related to the framework of G-L\'evy process as an upper-expectation.
Paczka then inaugurated the integrals and the It\^o formula based on the G-L\'evy process \cite{pa}. The existence and exponential estimates for solutions to stochastic differential equations (SDEs) driven by G-L\'evy process were established by Wang and Gao \cite{wg}. They also constructed the BDG-type inequality in the stated framework \cite{wg}. The existence theory for solutions to SDEs based on G-L\'evy process having discontinuous coefficients was given by  Wang and  Yuan \cite{wy}. The quasi-sure exponential
stability of SDEs in the framework of G-L\'evy process was initiated by Shen et. al. \cite{swy}. To the best of our knowledge no text is available on the study of existence-uniqueness and exponential estimates for solutions to stochastic functional differential equations (SFDEs) driven by G-L\'evy process. Consequently, the current research is concentrated on this theme.
Let $\R^d$ be the d-dimensional Euclidean space and $\R^d_0=\R^d \setminus \{0\}$.
Consider  $BC((-\infty,0]; \R^d)$, the family of bounded continuous $\R^d$-valued mappings $\psi$ defined on $(-\infty,0]$ with norm
$\|\psi\|= \sup_{-\infty<\theta\leq0}|\psi(\theta)|$ \cite{yx}. Let $\mathcal{F}_t=\sigma\{\mathcal{B}(v):0\leq v\leq t\}$ be the natural filtration defined on a complete probability space $(\mathcal{S},\mathcal{F},P)$. Assume that $\{\mathcal{F}_t:t\geq0\}$ assures the usual characteristics. Let $f:[0,T]\times BC((-\infty,0]; \R^d)  \rightarrow \R^d$, $g:[0,T]\times BC((-\infty,0]; \R^d)  \rightarrow R^{d\times m}$, $h:[0,T]\times BC((-\infty,0]; \R^d)  \rightarrow R^{d\times m}$ and $K:[0,T]\times BC((-\infty,0]; \R^d)  \rightarrow \R^{d\times m}$ be Borel measurable. We consider the following SFDE driven by G-L\'evy process:
\begin{equation}\label{1} dx(t)=f(t,x_t)dt+g(t,x_t)d<B,B>(t)+h(t,x_t)dB(t)+\int_{R^d_0}K(t,x_{t-},z)L(dt,dz), \end{equation}
on $t\in [0,T]$ with initial condition $\varsigma(0)\in \R^d$, $x_t=\{x(t+\theta),-\infty<\theta\leq 0\}$ and $x_{t-}$ indicates the left limits of $x_t$. $B(t)$ is a $d$-dimensional $G$-Brownian motion. For all $x\in \R^d$, $f(.,x), g(.,x), h(.,x)\in \mathbb{M}_G^2((-\infty,T];\R^d)$ and $K(.,x,.)\in \mathcal{H}^2_G((-\infty,T]\times \R^d_0;\R^d)$. Equation \eqref{1} has the following  initial condition.
\begin{equation}\label{i} x_0=\zeta=\{\zeta(\theta): -\infty<\theta\leq0\}, \end{equation}
is $\mathcal{F}_0$-measurable, $BC((-\infty,0]; \R^d)$-value random variable such that $\zeta\in \mathbb{M}_G^2((-\infty,T]; \R^d)$.

  The rest of the article is arranged as follows. Fundamental results and definitions of the G-framework are given in section 2. The existence-and-uniqueness  of solutions to SFDEs driven by G-L\'evy process is studied in section 3. Here the boundedness of solutions is determined. The error estimation between the exact and approximate solutions is shown. The exponential estimate for solutions to SFDEs driven by G-L\'evy process is constructed in section 4.

\section{Fundamental settings} In this section, we include preliminary results and notions of the G-framework required for the subsequent sections of this article \cite{dhp,f1,f2,p}. Consider $\mathcal{S}_T=C_0([0,T],\R^d)$, the space of real valued continuous mappings on $[0,T]$ such that $w(0)=0$ endowed with the
distance
\begin{equation*}\rho(w^1,w^2)=\sum_{i=1}^{\infty}\frac{1}{2^i}\Big(\max_{t\in[0,i]}|w^1(t)-w^2(t)|\wedge1\Big).\end{equation*}
Let for any $w\in \mathcal{S}_T$ and $t\geq 0$, $B(t,w)=w(t)$  be the
canonical process.  Let $\mathcal
{F}_t=\sigma\{B(v), 0\leq v \leq t\}$ be the filtration
generated by canonical process $\{B(t),t\geq0\}$ and $\mathcal
{F}=\{\mathcal {F}_t\}_{t\geq 0}$. For any $T>0$, define
$\mathcal {L}_{ip}(\mathcal{S}_T)=\{\phi(B(t_1),B(t_2),...,B(t_d)):d\geq1,t_{1},t_{2},...,t_{d}\in[0,T],\phi\in
C_{b.Lip}(\R^{d\times m}))\},$
 where $C_{b.Lip}(\R^{d\times m})$ is a space of bounded Lipschitz functions. A functional $\mathbb{E}$ defined on $\mathcal {L}_{ip}(\mathcal{S}_T)$ is known as a
 sublinear expectation if it ensures the characteristics given as follows. For every $x,y\in \mathcal {L}_{ip}(\mathcal{S}_T)$
 \begin{itemize}
 \item[${\bf(1)}$] Monotonicity: $\mathbb{E}[x]\geq \mathbb{E}[y]$  if $x\geq y$. 
 \item[${\bf(2)}$] Constant Preserving: For all $c\in \R$, $\mathbb{E}[c]=c$.
 \item[${\bf(3)}$] Sub-additivity: $\mathbb{E}[x]+\mathbb{E}[y]\geq \mathbb{E}[x+y]$.
  \item[${\bf(4)}$] Positive homogeneity: For all $\kappa>0$, $\mathbb{E}[\kappa x]=\kappa \mathbb{E}[x]$.
\end{itemize}
 For $t\leq
 T$,  $\mathcal {L}_{ip}(\mathcal{S}_t)\subseteq \mathcal {L}_{ip}(\mathcal{S}_T)$ and
$\mathcal {L}_{ip}(\mathcal{S})=\cup_{n=1}^{\infty}\mathcal {L}_{ip}(\mathcal{S}_n)$. For $p\geq 1$,
 $\mathcal {L}^p_{G}(\mathcal{S})$ indicates the completion of $\mathcal
{L}_{ip}(\mathcal{S})$ endowed with the
 Banach norm $\hat{\mathbb{E}}[|.|^p]^{\frac{1}{p}}$ and
$\mathcal {L}_{G}^p(\mathcal{S}_t)\subseteq \mathcal {L}_{G}^p(\mathcal{S}_T)\subseteq
\mathcal {L}_{G}^p(\mathcal{S})$ for $0\leq t\leq T <\infty$. The triple $(\mathcal{S},\mathcal {L}_{ip}(\mathcal{S}_T),\mathbb{E})$ is recognized as a sublinear expectation space.
For $p\geq 1$, a partition of $[0,T]$ is a finite order subset $\{\mathcal{A}_T^\N:\N\geq 1\}$ so that
$\mathcal{A}_T^\N: 0=t_0<t_1<...<t_{\N}=T\}$. The space $\mathbb
{M}^{p,0}_G([0,T])$, $p\geq 1$ of simple processes is defined by
\begin{equation}\label{p1}\mathbb
{M}^{p,0}_G([0,T])=\Big\{\eta_t(z)=\sum_{i=0}^{\N-1}\xi_{t_i}(z)I_{[t_i,t_{i+1}]}(t);\,\,\xi_{t_i}(z)\in
\mathcal {L}_G^p(\Omega_{t_{i}})\Big\}.\end{equation} The completion of space
\eqref{p1} equipped  with the norm
$\|\eta\|=\Big\{\int_0^T\mathbb{E}[|\eta(s)|^p]ds\Big\}^{1/p}$
is indicated by $\mathbb {M}_{G}^{p}(0,T),$ $p\geq 1$.
\begin{defn} Let $\eta_t\in
\mathbb {M}_G^{p}(0,T)$, $p\geq1$. Then the G-It\^{o}'s integral
 is defined by
\begin{align*}
\int_0^T\eta(s)dB(s)=\sum_{i=0}^{\N-1}\xi_i\Big(B({t_{i+1}})-B({t_i})\Big).\end{align*}
\end{defn}
\begin{defn} For a partition $0=t_0<t_1<...<t_{\N-1}=t$, the quadratic variation process $\{\langle
B\rangle(t)\}_{t\geq0}$ is defined by
\begin{align*}\begin{split}&
\langle
B\rangle(t)=\lim_{\N\rightarrow\infty}\sum_{i=0}^{\N-1}\Big(B({t_{i+1}^{\N}})-B({t_{i}^{\N}})\Big)^2
={B(t)}^2-2\int_0^tB(s)dB(s).
\end{split}\end{align*}
\end{defn}
A mapping $\Pi_{0,T}:\mathbb {M}^{0,1}_{G}(0,T)\mapsto \mathcal
{L}^2_G(\mathcal{F}_T)$  is given by
\begin{align*} \Pi_{0,T}(\eta)=\int_0^T\eta(s)d\langle B\rangle(s)=
\sum_{i=0}^{{\N}-1}\xi_i\Big(\langle
B\rangle_{{{(}}t_{i+1})}-\langle B\rangle({t_{i}})\Big).
\end{align*}
It can be extended to $\mathbb {M}^1_G(0,T)$ and for
$\eta\in \mathbb {M}^1_G(0,T)$ this is still given by
\begin{align*}
\int_0^T\eta(s)d\langle B\rangle(s)=\Pi_{0,T}(\eta).
\end{align*}
Let $\mathcal{Q}$ be a weakly compact set that represent $\mathbb{E}$.
The capacity $\hat{\nu}$ is given as the
following
\begin{equation*}\hat{\nu}(A)=\sup_{\mathbb{P}\in\mathcal{Q}}\mathbb{P}(A),\,\,\,\,\,A\in \mathcal{F}_T.\end{equation*}
The set $A$ is polar if $\hat{\nu}(A)=0$. A characteristic holds
quasi-surely (q.s) if it sustains outside a polar set.
\begin{lem}\label{l2} Let $x\in \mathcal{L}_G^p$ and $\hat{\mathbb{E}}|x|^p<\infty$. Then
\begin{equation*}\hat{\nu}(|x|>c)\leq \frac{\mathbb{E}[|x|^p]}{c},\end{equation*}
for any $c>0$.
\end{lem}
The proof of the lemmas \ref{l3} and \ref{l4} can be seen in
\cite{g}. 
\begin{lem}\label{l3} Let $\lambda\in \mathbb {M}_{G}^{p}(0,T)$, $p\geq 2$. Then
\begin{equation*}
\mathbb{E}\Big[\sup_{0\leq t\leq
T}\Big|\int_0^t\lambda(s)dB(s)\Big|^p\Big]\leq \alpha\mathbb{E}\Big[\int_0^t|\lambda(s)|^2ds\Big]^{\frac{p}{2}},
\end{equation*}
where $0<\alpha=k_2T^{\frac{p}{2}-1}<\infty$, $k_2$ is a positive constant depending on $p$.
\end{lem}

\begin{lem}\label{l4}
Let $\lambda\in \mathbb {M}_{G}^{p}(0,T)$, $p\geq 1$. Then
\begin{equation*}
\mathbb{E}\Big[\sup_{0\leq t\leq
T}\Big|\int_0^t\lambda(s)d\langle B,
B \rangle (s)\Big|^p\Big]\leq \beta\mathbb{E}\Big[\int_0^t|\lambda(s)|^2ds\Big]^{\frac{p}{2}},
\end{equation*}
where $0<\beta=k_1T^{p-1}<\infty$ and $k_1$ is a positive constant depending on $p$.
\end{lem}
\begin{defn} A stochastic process $\{x(t),t\geq0\}$ defined on a sublinear expectation space $(\mathcal{S},\mathcal {L}_{ip}(\mathcal{S}_T),\mathbb{E})$ is known as a
a G-L\'evy process if it ensures the upcoming five characteristics:
 \begin{itemize}
 \item[${\bf(1)}$] $x(t)=0$.
 \item[${\bf(2)}$] For any $t,s\geq 0$, the increment $x(t+s)-x(s)$ is independent of $x(t_1), x(t_2),...,x(t_n)$, $\forall$ $n\in \N$ and $0\leq t_1\leq t_2,...,\leq t_n\leq t$.
 \item[${\bf(3)}$] For every $s,t\geq 0$, the distribution $x(t+s)-x(s)$ does not depend on $t$.
  \item[${\bf(4)}$] For each $t\geq 0$ there exists a decomposition $x(t)=x^c(t)+x^d(t)$.
  \item[${\bf(5)}$] $(x^c(t),x^d(t))_{t\geq 0}$ is a $2d$-dimensional L\'evey process satisfying
   \begin{equation*} \lim_{t\downarrow0}\frac{\mathbb{E}[|x^c(t)|^3]}{t}=0,\,\,\,\,\,\,\,\,\mathbb{E}[|x^d(t)|]\leq \alpha t,\,\,\,\,\,\,\,\,t\geq 0,
\end{equation*}
 where $\alpha$ is a constant depends on $x$.
\end{itemize}
\end{defn}
If $\{x(t),t\geq0\}$ satisfies only the first three properties i.e. $1$-$3$, then it is the classical L\'evy process. It is known that $x^c(t)$
is  generalized G-Brownian motion and $x^d(t)$ is of finite variation, where $x^c(t)$ and $x^d(t)$ are continuous part and jump part respectively.
Let $\mathcal{H}^\delta_G([0,T]\times \R^d_0)$ be a collection of all basic fields defined on $[0,T]\times \R^d_0\times \mathcal{S}$ of the form
\begin{equation*}
K(u,z)(w)=\sum_{i=1}^{n-1}\sum_{j=1}^m \Lambda_{i,j}1_{(t_i,t_{i+1}]}(u)\psi_j(z),
\end{equation*}
where $n,m\in \N$ and $0\leq t_{1}< t_{2}.....< t_{n}\leq T$, $\{\psi_j\}_{j=1}^m\subset C_{b.lip}(\R^d)$ are
mappings with disjoint supports such that $\psi_j(0)=0$ and $\Lambda_{i,j}=\phi_{i,j}(x_{t_1},...,x_{t_i}-x_{t_{i-1}})$,  $\phi_{i,j}\in C_{b.lip}(\R^{d\times i})$.
The norm on this space is given by
\begin{equation*}
\|K\|_{\mathcal{H}^p_G([0,T]\times \R^d_0)}=\mathbb{E}\Big[\int_0^T \sup_{v\in\nu}\int_{\R^d_0}|K(s,z)|^pv(dz)ds\Big]^{\frac{1}{p}},\,\,\,\,\,\,p=1,2.
\end{equation*}
\begin{defn} The It\^o integral of $K\in \mathcal{H}^\delta_G([0,T]\times \R^d_0)$ w. r. t. jump measure $L$ is given as follows
\begin{equation*}
\int_0^t \int_{\R^d_0}K(s,z)L(ds,dz)=\sum_{v<s\leq t}K\Big(s,\bigtriangleup x(s)\Big),\,\,\,\,q.s.
\end{equation*}
where $0\leq v<t\leq T$.
\end{defn}
Let $\mathcal{H}^p_G([0,T]\times \R^d_0)$ be the topological completion of $\mathcal{H}^\delta_G([0,T]\times \R^d_0)$ under
the norm $\|K\|_{H^p_G([0,T]\times \R^d_0)}$, $p=1,2$. We can sill extend the It\^o integral to the space $\|K\|_{H^p_G([0,T]\times \R^d_0)}$, $p=1,2$, where the extended integral has valves in $\mathcal{L}^p_G(\mathcal{S}_T)$, $p=1,2$. For the above integrals, we have he following BDG-type inequality. For the proof see \cite{wg}.
\begin{lem}\label{ln} Let $K(s,z)\in \mathcal{H}^2_G([0,T]\times \R^d_0)$. Then a c$\grave{a}$dl$\grave{a}$g modification $\hat{x}(t)$ of $x(t)=\int_0^t \int_{\R^d_0}K(s,z)L(ds,dz)$
exists such that for all $t\in [0,T]$ and $p\geq2$
\begin{equation*}
\mathbb{E}\Big[\sup_{0\leq s\leq t}|\hat{x}(t)|^2\Big]\leq k_3\mathbb{E}\Big[\int_0^t \int_{\R^d_0}K^2(s,z)\nu(dz)ds\Big],
\end{equation*}
where $k_3$ is a positive constant depending on $T$.
\end{lem}
\section{Bounded-ness and existence-uniqueness results for SFDEs driven by G-L\'evy process} In this section, we shall determine the boundedness and existence-uniqueness results for solutions to problem \eqref{1}. Let us first see the definition of solutions to equation \eqref{1}.
\begin{defn} An $\mathcal{F}_t$-adopted c$\grave{a}$dl$\grave{a}$g process $x(t)\in \mathbb{M}_{G}^2((-\infty,T];\R^d)$ is called a solution to \eqref{1} with the initial data \eqref{i} if it satisfies
\begin{equation*}  x(t)= \zeta(0) +  \int _{0}^t f(s,x_s)ds+  \int _{0}^t g(s,x_s)d\langle B,B \rangle (s)+ \int _{0}^t h(s,x_s)dB(s) + \int _{0}^t \int_{\R_{0}^d}  K(s,x_s-,z)L(ds,dz).
\end{equation*}
A solution $x(t)$ of \eqref{1} is said to be unique if it is identical to any other solution $y(t)$ of the stated equation  i.e.
\begin{equation*}\mathbb{E}[|x(t)-y(t)|^2]=0,\end{equation*}
holds q.s.
\end{defn}
All through this article we propose the upcoming linear growth and Lipschitz conditions respectively.
\begin{itemize}
\item[${\bf(A_1)}$] For every $x\in BC((-\infty,0]; \R^d)$, a positive number $c_1 $ exists so that
\begin{equation*} \mid f(t,x)\mid ^2 \vee \mid g(t,x)\mid ^2 \vee \mid h(t,x)\mid ^2 \vee \int_{\R_{0}^d}  \mid K(t,x,z)\mid ^2 \upsilon (dz)\leq c_1 (1+ \mid x\mid ^2) \end{equation*}
\item[${\bf(A_2)}$] For all $x,y \in BC((-\infty,0]; \R^d)$, a positive number $c_2$ exists so that
\begin{equation*}\begin{split} &\mid f(t,y)- f(t,x)\mid ^2 \vee \mid g(t,y)- g(t,x)\mid ^2 \vee \mid h(t,y) - h(t,x)\mid ^2 \\&\vee \int_{\R_{0}^d}  \mid K(t,y,z)- K(t,x,z \mid ^2 \upsilon (dz)\leq c_2 \mid y-x\mid ^2.\end{split} \end{equation*}
\end{itemize}
In the forthcoming lemma we prove that any solution $x(t)$ of equation \eqref{1} is bounded, in particular $ x(t)\in \mathbb{M}_{G}^2 \Big((-\infty , T]; \R^d\Big) $.
\begin{lem}\label{lf2} Let $x(t)$ be a solution of equation \eqref{1} with initial data \eqref{i} such that $  \mathbb{E} \|x \|^2 \leq \infty$. Assume that the growth condition $A_1$ holds. Then
\begin{equation} \mathbb{E}\Big[\sup_{-\infty \leq s\leq t} \mid x(s)\mid ^2\Big] \leq \mathbb{E} \| \zeta \|^2 + 5(1+ c_1kT) e^{5c_1kT}, \end{equation}
where $k= (1+k_1)T+k_2+k_3$ and $k_1,k_2,k_3$ are positive constants.\end{lem}
\begin{proof} Consider equation \eqref{1} and use the basic inequality $|\sum _{i=1} ^5 a_i|^2 \leq 5 \sum _{i=1} ^5 \mid a_i|^2$ to derive
\begin{equation*}\begin{split} \mid x(t)\mid^2 &\leq 5 \mid \zeta(0)\mid^2 + 5\Big| \int _{0}^t f(s,x_s)ds\Big|^2+ 5\Big| \int _{0}^t g(s,x_s)d\langle B,B \rangle (s)\Big|^2+ \\& 5\Big| \int _{0}^t h(s,x_s)dB(s)\Big|^2 + 5\Big| \int _{0}^t \int_{\R_{0}^d}  K(s,x_s-,z)L(ds,dz)\Big|^2 .
\end{split}\end{equation*}
From the G- expectation, lemmas \ref{l3}, \ref{l4}, \ref{ln} and the cauchy inequality we get
\begin{equation*}\begin{split}
\mathbb{E}\Big[ \sup_{0 \leq s\leq t}  \mid x(s)\mid ^2\Big] &\leq 5\mathbb{E}\mid \zeta(0)\mid ^2 + 5t\mathbb{E}  \int _{0}^t |f(s,x_s)ds|^2+ 5k_1t\mathbb{E} \int _{0}^t |g(s,x_s)|^2 ds\\&
+ 5k_2 \mathbb{E} \int _{0}^t |h(s,x_s)|^2 ds + 5k_3\mathbb{E} \int _{0}^t \int_{\R_{0}^d}  |K(s,x_s-,z)|^2 \upsilon(dz)ds .
\end{split}\end{equation*}
In view of assumption $A_1$ we deduce that
\begin{equation*}\begin{split}
\mathbb{E}\Big[\sup _{0 \leq s\leq t} \mid x(s)\mid ^2\Big] &\leq  5\mathbb{E} \| \zeta\|^2 +5c_1(T+k_1T+k_2+k_3)T+
5c_1(T+k_1T+k_2+k_3)T\int _{0}^t \mathbb{E}|x_s|^2ds\\& \leq 5\mathbb{E} \| \zeta\| ^2
 +5c_1(T+k_1T+k_2+k_3)T \\& +5c_1(T+k_1T+k_2+k_3)\int _{0}^t \Big[\mathbb{E} \|
\zeta\| ^2+ \mathbb{E}\Big(\sup _{0 \leq u\leq s}|x(u)|^2\Big)\Big]ds \\& \leq 5\mathbb{E} \|
\zeta\| ^2+5c_1kT+ 5c_1kT\mathbb{E} \| \zeta\| ^2 + 5c_1k\int _{0}^t \mathbb{E}\Big[\sup
_{0 \leq u\leq s}|x(u)|^2\Big]ds
\end{split}\end{equation*}
where $\hat {k} =T+k_1T+k_2+k_3$. From the Grownwall inequality it follows
\begin{equation}\label{e1}
\mathbb{E}\Big[\sup _{0 \leq s\leq t}|x(s)|^2\Big]\leq5[(1+c_1 \hat{k}T) \mathbb{E}\| \zeta\| ^2+ c_1 \hat{k}T]e^{5c_1\hat{k}T}
\end{equation}
Noticing that
\begin{equation*}
\mathbb{E}\Big[\sup _{-\infty < s\leq t}|x(s)|^2\Big] \leq \mathbb{E} \| \zeta\|^2 + \mathbb{E}\Big[\sup _{0 \leq s\leq t}|x(s)|^2\Big],
\end{equation*}
it yields
\begin{equation*}
\mathbb{E}\Big[\sup _{-\infty \leq s\leq t}|x(s)|^2\Big] \leq \mathbb{E} \| \zeta\|^2 + 5[(1+c_1 \hat{k}T) \mathbb{E}\| \zeta\| ^2+ c_1 \hat{k}T]e^{5c_1\hat{k}T}.
\end{equation*}
Letting $t=T$, we deduce the desired expression.
\end{proof}
For $t\in [0,T]$, define $x^0(t)=\zeta(0)$ and $x_{0}^0=\zeta$. For
each $n=1,2,...$, we set $x_{0}^n=\zeta$ and define
the Picard iteration,
\begin{equation}\begin{split}\label{2}  x^n(t)&=\zeta(0)+\int_{0}^t
f(s,x_{s}^{n-1})ds+  \int_{0}^t g(s,x_{s}^{n-1})d \langle B,B
\rangle (s)+\int_{0}^t h(s,x_{s}^{n-1})dB(s)\\&+\int_{0}^t
\int_{\R_{0}^d} K(s,x_{s-}^{n-1},z)L(ds,dz)\,\,\,\,\,\,t\in [0,T].\end{split}\end{equation}
Next we prove the existence-uniqueness result and error estimation between the exact solution $x(t)$ and Picard approximate solutions $x^n(t),n\geq 1$.
\begin{thm}\label{lf1} Let assumptions $A_1$ and $A_2$ hold and $\mathbb{E} \|\zeta \|^2 <
\infty$. Then equation \eqref{1} admits a unique c$\grave{a}$dl$\grave{a}$g solution $ x(t)\in \mathbb{M}_{G}^2 ((-\infty , T]; \R^d) $.
Moreover, for all $n\geq1$, the Picard approximate solutions $x^n(t)$ and unique exact solution $x(t)$ of \eqref{1} satisfy that
\begin{equation*}
\begin{split}
\mathbb{E}\Big[\sup_{0\leq s\leq t} |x^n(s)-x(s)|^2\Big]\leq \frac{C[Mt]^n}{n!}e^{Mt},
\end{split}
\end{equation*}
where $C=4c_2[(T+Tk_1+k_2+k_3)(1+\mathbb{E}\|
\zeta\|^2)T, M=4c_2[(T+Tk_1+k_2+k_3)$ and $k_1, k_2, k_3$ are positive constants.
\end{thm}
\begin{proof} Consider the Picard iteration sequence $\{x^n,n\geq1\}$ given by \eqref{2}. Obviously $x^0(t)\in \mathbb{M}_{G}^2\Big((-\infty,T];\R^d\Big)$.
From the fundamental inequality $|\sum_{i=1}^5 a_i|^2 \leq
5\sum_{i=1}^5 |a_i|^2 $, lemmas \ref{l3}, \ref{l4}, \ref{ln},
the cauchy inequality and assumption $A_1$, we deduce
\begin{equation*}\begin{split} \mathbb{E}\Big[\sup_{0\leq s\leq t}|x^n(s)|^2\Big]
&\leq
5\mathbb{E}\|\zeta\|^2+5c_1 \hat{k} T + 5c_1 \hat{k} T\mathbb{E}\|\zeta\|^2 + 5c_1
\hat{k} \int_{0}^t\mathbb{E}\Big[\sup_{0\leq u\leq
s}|x^{n-1}(u)|^2\Big]ds,\end{split}\end{equation*}
where $\hat {k}=(1+k_1)T+k_2+k_3.$ Noticing that
\begin{equation*}\begin{split}
\max_{1\leq n\leq j} \mathbb{E}\Big[\sup_{0\leq s\leq t}|x^{n-1}(s)|^2\Big] &\leq
\max\Big\{\mathbb{E}\| \zeta\|^2,\max_{1\leq n\leq j} \mathbb{E}\Big[\sup_{0\leq s\leq t}
|x^n(s)|^2\Big]\Big\}\\&\leq \mathbb{E}\| \zeta\|^2+\max_{1\leq n\leq j}
\mathbb{E}\Big[\sup_{0\leq s\leq t} |x^n(s)|^2\Big],
\end{split}\end{equation*}
it follows
\begin{equation*}\max_{1\leq n\leq j}
\mathbb{E}\Big[\sup_{0\leq s\leq t} |x^n(s)|^2\Big] \leq 5\mathbb{E}\| \zeta\|^2 +5c_1
\hat{k}T+ 10 c_1 \hat{k}T\mathbb{E}\| \zeta\|^2+5c_1 \hat{k}\int_{0}^t
\max_{1\leq n\leq j} \mathbb{E}\Big[\sup_{0\leq u\leq s} |x^n(u)|^2\Big]ds.
\end{equation*}
The Grownwall inequality yields
\begin{equation*}
\max_{1\leq n\leq j}
\mathbb{E}\Big[\sup_{0\leq s\leq t} |x^n(s)|^2\Big]\leq 5[(1+2c_1 \hat{k}T)\mathbb{E}\|
\zeta\|^2+c_1 \hat{k}T]e^{c_1 \hat{k}t},
\end{equation*}
but $j$ is arbitrary and letting $t=T$ it follows
\begin{equation}\label{19}
\mathbb{E}\Big[\sup_{0\leq s\leq T} |x^n(s)|^2\Big]\leq 5[(1+2c_1 \hat{k}T)\mathbb{E}\|
\zeta\|^2+c_1 \hat{k}T]e^{c_1 \hat{k}T}.
\end{equation}
From the sequence $\{x^n(t);t\geq 0\}$ defined by
\eqref{2}, we have
\begin{equation*}\begin{split} x^1(t)-x^0(t)&=\int_{0}^t
f(s,x_{s}^{0})ds +  \int_{0}^t g(s,x_{s}^{0})d \langle B,B \rangle
(s)+\int_{0}^t h(s,x_{s}^{0})dB(s)\\& +\int_{0}^t \int_{\R_{0}^d}
K(s,x_{s-}^{0},z)L(ds,dz).\end{split}\end{equation*}
In view of G-expectation, lemmas \ref{l3}, \ref{l4}, \ref{ln}, the Cauchy inequality and assumption $A_1$, we derive
\begin{equation*}
\begin{split}
&\mathbb{E}\Big[\sup_{0\leq s\leq t} |x^1(s)-x^0(s)|^2\Big] \leq 4c_1[(T+Tk_1+k_2+k_3)\int_{0}^t(1+\mathbb{E}\|
\zeta\|^2]ds
\leq C,
\end{split}
\end{equation*}
where $C=4c_1[(T+Tk_1+k_2+k_3)(1+E\|
\zeta\|^2)T$. Next by similar arguments and assumption $A_2$, it follows
\begin{equation*}
\begin{split}
&\mathbb{E}\Big[\sup_{0\leq s\leq t} |x^2(s)-x^1(s)|^2\Big] \leq 4c_2(T+Tk_1+k_2+k_3)\int_{0}^t\mathbb{E}|x_{s}^{1}-x_{s}^{0}|^2ds\\&
\leq 4c_2(T+Tk_1+k_2+k_3)\int_{0}^t\mathbb{E}\Big[\sup_{0\leq v\leq s} |x^1(v)-x^0(v)|^2\Big]ds\\&
\leq 4c_2(T+Tk_1+k_2+k_3)Ct.
\end{split}
\end{equation*}
Similarly, we derive
\begin{equation*}
\begin{split}
&\mathbb{E}\Big[\sup_{0\leq s\leq t} |x^3(s)-x^2(s)|^2\Big] \leq
 C[4c_2[(T+Tk_1+k_2+k_3)]^2\frac{t^2}{2!}.
\end{split}
\end{equation*}
Thus for all $n\geq 0$, we claim that
\begin{equation}
\begin{split}\label{3}
\mathbb{E}\Big[\sup_{0\leq s\leq t}|x^{n+1}(s)-x^{n}(s)|^2\Big]ds \leq C\frac{[Mt]^n}{n!},
\end{split}
\end{equation}
where $C=4c_2[(T+Tk_1+k_2+k_3)(1+E\|
\zeta\|^2)T$ and $M=4c_2(T+Tk_1+k_2+k_3)$. With the mathematical induction we verify that for all $n\geq 0$, \eqref{3} holds. For $n=0$, it has been proved. Let \eqref{3} holds for some $n\geq0$. By using similar arguments as above, we derive
\begin{equation*}
\begin{split}
\mathbb{E}\Big[\sup_{0\leq s\leq t}|x^{n+2}(s)-x^{n+1}(s)|^2\Big]ds &\leq 4c_2(T+Tk_1+k_2+k_3)\int_{0}^t \mathbb{E}|x_{s}^{n+1}-x_{s}^{n}|^2ds \\
&\leq M\int_{0}^t \mathbb{E}\Big[\sup_{0\leq v\leq s}|x^{n+1}(v)-x^{n}(v)|^2\Big]ds\\
&\leq M\int_{0}^t \frac{C[Mt]^n}{n!}ds
\leq C[M]^{n+1} \frac{t^{n+1}}{(n+1)!}=\frac{C[Mt]^{n+1}}{(n+1)!}.
\end{split}
\end{equation*}
This shows that \eqref{3}  holds for $n+1$. Thus by induction \eqref{3}  holds for all $n\geq0$.
By virtue of lemma \ref{l2} we acquire
\begin{equation*}
\begin{split}
&\hat{\nu}\Big\{ \sup_{0\leq s\leq T}|x^{n+1}(s)-x^{n}(s)|^2>\frac{1}{2^n}\Big\}\leq
2^n\mathbb{E}\Big[\sup_{0\leq s\leq T}|x^{n+1}(s)-x^{n}(s)|^2\Big]\leq \frac{K[2Mt]^n}{n!}.
\end{split}
\end{equation*}
Since $\sum_{n=0}^{\infty} \frac{K[2Mt]^n}{n!}<\infty$, from the Borel-Cantelli lemma we get that for almost all $w$ a positive integer $n_{0}=n_{0}(w)$ exists  so that
\begin{equation}
\sup_{0\leq t\leq T}|x^{n+1}(t)-x^{n}(t)|^2\leq \frac{1}{2^n},\,\, \text{as}\,\, n\geq n_{0}.
\end{equation}
It implies that q.s. the partial sums
\begin{equation*}
x^{0}(t)+\sum_{i=0}^{n-1}[x^{i+1}(t)-x^{i}(t)]=x^n(t),
\end{equation*}
are uniformly convergent on $t\in(-\infty,T]$. Denote the limit by $x(t)$. Then the sequence $x^{n}(t)$ converges uniformly to $x(t)$ on $t\in (-\infty,T]$.
It follows that $x(t)$ is $\mathcal{F}_t$-adapted and c$\grave{a}$dl$\grave{a}$g. Also, from \eqref{3}, we can see that $\{x^n(t):n\geq1\}$ is a cauchy sequence in $\mathcal{L}_{G}^{2}$. Hence $x^n(t)$ converges to $x(t)$ in $\mathcal{L}_{G}^{2}$, that is,
\begin{equation*}
\mathbb{E}|x^{n}(t)-x(t)|^2\rightarrow 0,\,\, \text{as}\,\, n\rightarrow \infty.
\end{equation*}
Taking limits $n\rightarrow \infty$ from \ref{19}, we deduce
\begin{equation}\label{20}
\mathbb{E}\Big[\sup_{0\leq s\leq T}|x(s)|^2\Big]\leq 5[(1+2c_1\hat{k}T)\mathbb{E}\|
\zeta\|^2+c_1\hat{k}T]e^{c_1\hat{k}T}.
\end{equation}
Next we need to verify that $x(t)$ satisfies equation \eqref{1}. In view of assumption $A_2$ and using similar arguments as above, we derive
\begin{equation*}
\begin{split}
&\mathbb{E}\Big[\sup_{0\leq s\leq T}\int_{0}^{t}\Big|[f(s,x_{s}^{n})-f(s,x_s)]ds\Big|^2\Big]+\mathbb{E}\Big[\sup_{0\leq s\leq T}\int_{0}^{t}\Big|[g(s,x_{s}^{n})-g(s,x_s)]d\langle B,B\rangle (s)\Big|^2\Big]\\&+\mathbb{E}\Big[\sup_{0\leq s\leq T}\int_{0}^{t}\Big|[h(s,x_{s}^{n})-h(s,x_s)]dB(s)\Big|^2\Big]+\mathbb{E}\Big[\sup_{0\leq s\leq T}\int_{0}^{t}\int_{\R_{0}^{d}}\Big|[K(s,x_{s-}^{n},z)-K(s,x_{s-},z)]L(ds,dz)\Big|^2\Big]\\&
\leq T \int_{0}^{t}\mathbb{E}|f(s,x_{s}^{n})-f(s,x_s)|^2ds+Tk_1\int_{0}^{t}\mathbb{E}|g(s,x_{s}^{n})-g(s,x_s)|^2ds\\&
+k_2\int_{0}^{t}\mathbb{E}|h(s,x_{s}^{n})-f(s,x_s)|^2ds+k_3\int_{0}^{t}\int_{\R_{0}^{d}}\mathbb{E}|K(s,x_{s-}^{n},z)-K(s,x_{s-},z)|^2v(dz)ds\\&
\leq c_2(T+Tk_1+k_2+k_3)\int_{0}^{t}\mathbb{E}[\sup_{0\leq v\leq s}|x^{n}(v)-x(v)|^2]ds\rightarrow 0 \,\, \text{as},\,\, n\rightarrow \infty,
\end{split}
\end{equation*}
in other words
\begin{equation*}
\begin{split}
&\int_{0}^{t}f(s,x_{s}^{n})\rightarrow \int_{0}^{t}f(s,x_s),\,\, \text{in}\,\, \mathcal{L}_{G}^{2},\\
&\int_{0}^{t}h(s,x_{s}^{n})d\langle B,B\rangle(s)\longrightarrow \int_{0}^{t}h(s,x_s)d\langle B,B\rangle(s), \,\, \text{in}\,\, \mathcal{L}_{G}^{2},\\
&\int_{0}^{t}g(s,x_{s}^{n})\longrightarrow \int_{0}^{t}g(s,x_s), \,\, \text{in}\,\, \mathcal{L}_{G}^{2},\\
&\int_{0}^{t}h(s,x_{s}^{n})dB(s)\longrightarrow \int_{0}^{t}h(s,x_s)dB(s), \,\, \text{in}\,\, \mathcal{L}_{G}^{2},\\
&\int_{0}^{t}\int_{\R_{0}^{d}}^{ }K(s,x_{s-}^{n-1},z)L(ds,dz)ds\longrightarrow \int_{0}^{t}\int_{\R_{0}^{d}}K(s,x_{s-},z)L(ds,dz), \,\, \text{in}\,\, \mathcal{L}_{G}^{2}.
\end{split}
\end{equation*}
For $t \in [0,T]$ taking limits $n\rightarrow\infty$ in \eqref{2} we derive
\begin{equation*}
\begin{split}
\lim_{n\longrightarrow\infty} x^{n}(t)&=\zeta(0)+\int_{0}^{t}\lim_{n\longrightarrow\infty}f(s,x_{s}^{n-1})ds+\int_{0}^{t}\lim_{n\longrightarrow\infty}g(s,x_{s}^{n-1})d\langle B,B\rangle (s)+\int_{0}^{t}\lim_{n\longrightarrow\infty}h(s,x_{s}^{n-1})dB(s)\\&
+\int_{0}^{t}\lim_{n\longrightarrow\infty}K(s,x_{s-}^{n-1},z)L(ds,dz),
\end{split}
\end{equation*}
which yields
\begin{equation*}
x(t)=\zeta(0)+\int_{0}^{t}f(s,x_s)ds+\int_{0}^{t}g(s,x_s)d\langle B,B\rangle(s) +\int_{0}^{t}h(s,x_s)dB(s)+\int_{0}^{t}\int_{\R_{0}^{d}}^{ }K(s,x_s-,z)L(ds,dz),
\end{equation*}
$t\in [0,T]$. This show that $x(t)$ is the solution of \eqref{1}. To prove the uniqueness let us assume that equation \eqref{1} admits two solutions $x(t)$ and $y(t)$. Following similar arguments we derive
\begin{equation*}
\begin{split}
&\mathbb{E}\Big[\sup_{0\leq s\leq t}|y(s)-x(s)|^2\Big] \leq 4t \int_{0}^t\mathbb{E}|f(s, y_s)-f(s, x_s)|^2
ds+
 4k_1t \int_{0}^t\mathbb{E}|g(s, y_s)-g(s, x_s)|^2 ds \\& +4k_2 \int_{0}^t\mathbb{E}|f(s, y_s)-f(s, x_s)|^2 ds+ 4k_3\int_{0}^t
\int_{\R_{0}^d} \mathbb{E}|K(s,y_{s-},z)-K(s,x_{s-},z)|^2\upsilon (dz)ds
\end{split}
\end{equation*}
By virtue of assumption $A_2$, we deduce
\begin{equation*}
\begin{split}
\mathbb{E}\Big[\sup_{0\leq s\leq t}|y(s)-x(s)|^2\Big]
&\leq 4c_2(T+k_1T+k_2+k_3)\int_{0}^t \mathbb{E}\Big[\sup_{0\leq u\leq s}|y(u)-x(u)|^2\Big]ds.
\end{split}
\end{equation*}
From the Grownwall inequality and same initial data one can derive
\begin{equation}
\mathbb{E}\Big[\sup_{-\infty< s\leq t}|y(s)-x(s)|^2\Big]=0
\end{equation}
which means $x(t)=y(t)$ quasi-surely, for all
$t\in(-\infty,T]$. Finally, we have to prove the error estimation.
From equations \eqref{1} and \eqref{2}, using similar arguments as earlier it follows
\begin{equation*}
\begin{split}
&\mathbb{E}\Big[\sup_{0\leq s\leq t} |x^n(s)-x(s)|^2\Big]\leq 4T\int_{0}^{t}\mathbb{E}\Big|f(s,x_{s}^{n})-f(s,x_s)\Big|^2ds+4k_1T\int_{0}^{t}\mathbb{E}\Big|g(s,x_{s}^{n})-g(s,x_s)\Big|^2ds\\&
+4k_2\int_{0}^{t}\mathbb{E}\Big|h(s,x_{s}^{n})-h(s,x_s)\Big|^2ds+4k_3\int_{0}^{t}\int_{\R_{0}^{d}}\mathbb{E}\Big|K(s,x_{s-}^{n},z)-K(s,x_{s-},z)\Big|^2v(dz)ds\\&
 \leq 4c_2[(T+Tk_1+k_2+k_3)\int_{0}^{t}\mathbb{E}\Big[\sup_{0\leq v\leq s} \Big|x^n(v)-x(v)\Big|^2\Big]ds\\&
 \leq M\int_{0}^{t}\mathbb{E}\Big[\sup_{0\leq v\leq s} \Big|x^n(v)-x^{n-1}(v)\Big|^2\Big]ds+M\int_{0}^{t}\mathbb{E}\Big[\sup_{0\leq v\leq s} \Big|x^{n-1}(v)-x(v)\Big|^2\Big]ds.
\end{split}
\end{equation*}
In view of \eqref{3}, we obtain
\begin{equation*}
\begin{split}
\mathbb{E}\Big[\sup_{0\leq s\leq t} |x^n(s)-x(s)|^2\Big]\leq \frac{C[Mt]^n}{n!}+M\int_{0}^{t}\mathbb{E}\Big[\sup_{0\leq v\leq s} |x^{n}(v)-x(v)|^2\Big]ds
\end{split}
\end{equation*}
Consequently,
\begin{equation*}
\mathbb{E}\Big[\sup_{0\leq s\leq t} |x^n(s)-x(s)|^2\Big]\leq \frac{C[Mt]^n}{n!}e^{Mt},
\end{equation*}
which yields the error estimation between the Picard approximate solutions $x^n(t)$, $n\geq0$ and exact solution $x(t)$ of problem \eqref{1}.
\end{proof}
\section{Exponential estimates for SFDEs driven by G-Levy process}
To show the exponential estimates let us assume that problem \eqref{1} has a unique solution $x(t)$ on $t \in[0,\infty)$.  Now we derive the exponential estimate for \eqref{1} as follows.
\begin{thm} Let assumptions $A_1$ and $A_2$ hold. Then,
\begin{equation*}
\lim_{n\longrightarrow\infty} sup\frac{1}{t}\log|y(t)|\leq \frac{5}{2}c_1k,
\end{equation*}
where $k= (1+k_1)T+k_2+k_3$ and $k_1$, $k_2$, $k_3$ are positive constants.
\end{thm}
\begin{proof}
From assertion \eqref{20}, we know that
\begin{equation}\label{e2}
\mathbb{E}\Big[\sup_{0\leq s\leq T}|x(s)|^2\Big]\leq 5[(1+2c_1\hat{k}T)\mathbb{E}\|
\zeta\|^2+c_1\hat{k}T]e^{c_1\hat{k}T}.
\end{equation}
In view of \eqref{e2}, for each $m=1,2,3....,$ we have
\begin{equation*}
\begin{split}
\mathbb{E}\Big[\sup_{m-1\leq t\leq m}|x(t)|^2\Big]\leq 5[(1+2c_1\hat{k}T)\mathbb{E}\|
\zeta\|^2+c_1\hat{k}T]e^{c_1\hat{k}m}.
\end{split}
\end{equation*}
By using lemma \ref{l2} for any $\epsilon>0$, we derive
\begin{equation*}
\begin{split}
\hat{\nu}\Bigg\{w: \sup_{m-1\leq t\leq m} |x(t)|^2>e^{(5c_1k+\epsilon)m}\Bigg\}&\leq \frac{\mathbb{E}\Big[\displaystyle\sup_{m-1\leq t\leq m}|x(t)|^2\Big]}{e^{(5c_1k+\epsilon)m}}
\leq 5[(1+c_1\hat{k}T)\mathbb{E}\|
\zeta\|^2+c_1\hat{k}T]e^{-\epsilon m}.
\end{split}
\end{equation*}
Since the series $\sum_{m=1}^{\infty}5[(1+c_1\hat{k}T)\mathbb{E}\|
\zeta\|^2+c_1\hat{k}T]e^{-\epsilon m}$ is convergent, from the Borel-Cantelli lemma we derive tht for almost all $w\in \Omega$, a random integer $m_0=m_0(w)$ exists so that
\begin{equation*}
\begin{split}
\sup_{m-1\leq t\leq m}|x(t)|^2\leq e^{(5c_1k+\epsilon)m},\,\, \text{as}\,\, m\geq m_0.
\end{split}
\end{equation*}
This implies that for $m-1\leq t\leq m$ and $m\geq m_0$ we have
\begin{equation*}
\begin{split}
|x(t)|\leq e^{\frac{1}{2}}{(5c_1k+\epsilon)m}.
\end{split}
\end{equation*}
Hence,
\begin{equation*}
\begin{split}
\lim_{t\rightarrow\infty} sup\frac{1}{t}\log|x(t)|\leq \frac{1}{2}(5c_1\hat{k}+\epsilon),
\end{split}
\end{equation*}
the desired expression follows because $\epsilon$ is arbitrary.
\end{proof}

\end{document}